\documentclass[12pt]{amsart}

\usepackage
{amsmath,amsfonts,latexsym,graphicx,amssymb,amsthm}
\usepackage{tikz}
\usetikzlibrary{calc}

\newcommand{\HP}{\mathbb{H}}

\newcommand{\Q}{\mathbb{Q}}
\newcommand{\R}{\mathbb{R}}
\newcommand{\Z}{\mathbb{Z}}

\newcommand{\arcosh}{\mathrm{arcosh}}
\newcommand{\Area}{\mathrm{Area}}

\newcommand{\Card}{\mathrm{Card\,}}

\renewcommand{\fill}{\mathrm{fill}}
\newcommand{\Fill}{\mathrm{Fill}}
\newcommand{\Gal}{{\mathrm {Gal}}}
\newcommand{\GL}{{\mathrm {GL}}}

\newcommand{\Isom}{\mathrm{Isom}}

\newcommand{\Ker}{\mathrm{Ker}}

\newcommand{\Mat}{\mathrm{Mat}}

\newcommand{\Min}{\mathrm{Min}}

\newcommand{\PGL}{{\mathrm {PGL}}}
\newcommand{\PSL}{{\mathrm {PSL}}}
\newcommand{\Quat}{{\mathrm {Quat}}}

\newcommand{\SL}{{\mathrm {SL}}}

\newcommand{\Sol}{{\mathrm {Sol}}}

\newcommand{\Syst}{{\mathrm {Syst}}}
\newcommand{\Tr}{{\mathrm {Tr}}}

\renewcommand{\mod}{\,\mathrm{mod}\,}

\newtheorem{lemma}{Lemma}
\newtheorem{thm}[lemma]{Theorem}
\newtheorem{Schmutz}[lemma]{Schmutz Theorem}
\newtheorem{AII}[lemma]{An-Irhinger-Irmer Theorem}
\newtheorem{cor}[lemma]{Corollary}

\newcommand{\cH}{\mathcal{H}}

\newcommand{\cP}{\mathcal{P}}
\newcommand{\cS}{\mathcal{S}}
\newcommand{\cT}{\mathcal{T}}

\author{Olivier Mathieu}

\address{CNRS, Institut Camille Jordan, Universit\'e de Lyon, France}

\address{SUSTech Insternational Center for Mathematics, Shenzhen, China}

\email{mathieu@math.univ-lyon1.fr}

\title{Filling Surfaces with very few systoles}

\date{}

\begin{document}

\begin{abstract} In the paper we describe hyperbolic surfaces filled by their systoles, where the total number of systoles is in $O(\frac{g}{\ln \,g})$, that is
equivalent  to the lower bound of Anderson, Parlier and Pittet
\cite{APP}. 
Various papers \cite{SS}\cite{FB20}\cite{Sanki}\cite{ IM}\cite{ Mathieu} have investigated the same question, and the best previously known upper bounds where in $o(\frac{g}{{\sqrt{\ln \,g}}})$. 

Surprizingly  the present approach is, in our opinion, much simpler than the 
methods of earlier papers.
\end{abstract}

\maketitle

\section*{Introduction}
We start with  definitions. For $g\geq 2$, let  $\cT_g$ be the Teichm\"uller space, that is the space of all marked closed oriented hyperbolic surfaces of genus $g$. 

 Let $C$ be a finite set of geodesics on some surface $\Sigma\in \cT_g$,
where by {\it geodesic} we mean a closed geodesic.
 We say that $C$ {\it fills} the surface if $C$ cuts $\Sigma$ into polygons.
We denote as $\Syst(\Sigma)$  the set of all
{\it systoles} of $\Sigma$, that is the geodesic of shortest length.
Since the spectrum length is discrete \cite{Buser},  $\Syst(\Sigma)$ is always finite.

A generic surface  $\Sigma\in \cT_g$ admits a unique systole,  
which does not fill. However, W. Thurston had shown that the subspace $\cP_g\subset \cT_g$ of all  surfaces $\Sigma$  which are filled by their systoles is nonempty \cite{T85}. Thus, we define two  {\it filling numbers}
$\Fill(g)\geq \fill(g)$, as follows:

\begin{enumerate}
\item[(a)] $\Fill(g)=\Min_{\Sigma\in \cP_g}\,\Card\,\Syst(\Sigma), \text{ and}$

\item[(b)]$\fill(g)$ is the smallest integer $n$ for which there is 
a hyperbolic surface  $\Sigma\in\cT_g$ filled   by a set of $n$ systoles.
\end{enumerate}

In \cite{APP}, Anderson, Pittet and Parlier found the lower bound
 
\begin{align}
\fill(g)\geq (1-o(1)) \pi \frac{g}{\ln g} \label{APPbound}
\end{align}

\noindent for some explicit function $o(1)$. In this
paper, we provide an upper bound which is equivalent to the lower bound:

\begin{thm}\label{main} We have

$$\Fill(g)<9\ln(2+\sqrt{3}) \frac{g}{\ln g}=11.8526\ldots \frac{g}{\ln g},$$

for infinitely many genera $g\geq 2$.
\end{thm}

\subsection*{{ Previous related works}}

It has been known for some time that  $\Fill(g)\leq 2g$ for all $g$,
see \cite{SS},\cite{APP} and \cite{Sanki}. Indeed,   some of these references concerns the inequality  
$\fill(g)\leq 2g$.  

There are no better upper  bound known for all genera g, but there are
bounds valid for infinitely many $g$ in \cite{FB20}\cite{IM}\cite{Mathieu}.
Even the best bound 
 
\begin{align}
\Fill(g)\leq o(1) \frac{g}{\sqrt{\ln\,g}}\label{IMbound}
\end{align}
 
 \noindent for infinitely many $g$ is not equivalent to
Anderson, Parlier and Pittet lower bound (\ref{APPbound}). 

In fact, these upper bounds are not efficient: Fortier-Bourque bound
\cite{FB20} is not explicit and even
the explicit function  $o(1)$ in Inequality (\ref{IMbound}) has a very slow decay.
For example, in \cite{FB20} Fortier-Bourque  had shown that 
\begin{align}
\Fill(g)<2g \text{ for infinitely many } g. \label{2gbound}
\end{align}
This remarkable counterintuive result means that
 that the systoles can fill a surface although
the corresponding cycles  do not generate $H_1(\Sigma, \Q)$.
However, finding an explicit $g$ satisfying (\ref{2gbound}) has been challenging: the paper
 \cite{IM} only predicts the existence of such genus of size around $2^{1000}$. However  the analysis of an explicit example in \cite{Mathieu} shows 
that $\Fill(17)\leq 32$ and it was hinted that $\Fill(17)\leq 25$.

\subsection*{ Idea of the proof}

We start with the unique genus two surface $S$ which admits a tessalation by
four regular rigth-angled hexagons.  Although no name is attached to $S$, this surface has been widely studied since a long time,
see e.g.  \cite{Broughton}\cite{SS}\cite{Harvey}\cite{KuuNaa}.

The one-skeleton of the tessalation of $S$ consists of six geodesics
$c_1,\ldots,c_6$. It has been proved by P. Schmutz Schaller \cite{SS} that 
these geodesics are the six systoles of $S$. Their length is
$$L=2\arcosh(2)=\ln(2+\sqrt{3}).$$

It has been observed in \cite{KuuNaa} that $S$ is an arithmetic surface, see
also \cite{Takeuchi}.
We explicitly compute the corresponding quaternion ring.
Let $\cH$ be the  ring  defined by

$$\langle I, J\mid I^2=3, J^2=3, IJ+JI=0\rangle.$$

\noindent We  show that $\pi_1(S)$ is
isomorphic to an explicit subgroup $\Pi\subset \cH^*$ such that
$\cH^*=\Pi\times \{\pm1\}$. It follows that
$S\simeq \HP/\Pi$.

For $n\geq 2$, we define  integers $p_n$ and $q_n$ by the equation
$$(2+\sqrt{3})^n=p_n+q_n\sqrt{3},$$
\noindent and we consider the  surfaces 
$S(n)=\HP/\Pi(n)$ where 
$$\Pi(n):=\Pi\cap \big(\Z\oplus q_n\cH)$$
\noindent is a congruence subgroup of $\Pi$. Let $\pi_n: S(n)\to S$ be the 
Galois covering  with Galois group $\Pi/\Pi(n)$.

The key Lemma \ref{key} shows that the systoles of $S(n)$  are the connected components
of $\pi_n^{-1}(c_1)\cup\ldots\cup \pi_n^{-1}(c_6)$ and their lengths are
$nL$. We conclude that
$$\Card \Syst(S(n))=\frac{6}{n}[\Pi:\Pi(n)],$$
\noindent which easily implies that 

\begin{align}
\Card \Syst(S(n))\leq 18\ln(2+\sqrt{3}) \frac{g_{S(n)}}{\ln g_{S(n)}}, \label{inequation18}
\end{align}

\noindent where $g_{S(n)}$ is the genus of $S(n)$.

To finish the proof, we need to replace the constant $18$ by $9$.
We first observe that $S(2)$ is the  Galois covering
of $S$ with Galois groups $(\Z/2\Z)^4$, that is the surface of genus $17$ investigated in \cite{Mathieu}. When $n$ is even,
we have $\Pi(n)\subset\Pi(2)$, thus there is a Galois covering $\theta_n: S(n)\to S(2)$.

For any filling subset $A\subset\Syst(S(2))$, the set 

$$B:=\{c\in \Syst(S(n))\mid \theta(c)\in A\}$$
 
\noindent obviously fills $S(n)$, and its cardinality is 
$$\frac{\Card A}{\Card \Syst(S(2))} \Card \Syst(S(n))$$ 

\noindent It follows from a result on Penner systems
proved in \cite{Mathieu},   that  $\Syst(\Sigma)=B$
for some   hyperbolic surface 
$\Sigma$ near $S(n)$ in $\cT_g$. 

Therefore, we are looking at a filling set $A\subset \Syst(S(2))$ 
of minimal cardinality. In \cite{Mathieu}, we described
a filling set $A$ with $32$ systoles, and it was
hinted that $A$ can be further reduced to $25$ systoles:
this would  allow to replace $18$ by
$9.375=\frac{25}{48} 18$.
Recently  N. An, F. Ihringer and I. Irmer
got a better result  in  \cite{AII}:   $\Syst(S(2))$ contains 
a filling subset with only $24$ elements. This allows to  obtain the desired constant $9$.

\subsection*{Acknowledgements}
We thank Jean-Claude Sikorav, Don Zagier and Efim Zelmanov for interesting discussions.  We also thank Leonid Rivkin and Alessandra Frabetti
for their help concerning the figures.

\section{Conventions}

We denote as $\HP$ the Poincar\'e half-plane.
Its isometry group is isomorphic to
$\PGL_2(\R)$. Indeed an element
$$g=\begin{pmatrix} a&b\\c&d
\end{pmatrix}\in \GL_2(\R)$$
\noindent acts on $\HP$ as
$$g(z)=\frac{az+b}{cz+d}$$
\noindent if $\det g >0$ and as
$$g(z)=\frac{a\overline{z}+b}{c\overline{z}+d}$$
\noindent otherwise. The subgroup $\Isom^+(\HP)$ of preserving orientation isometries is isomorphic to 
$\PSL_2(\R)\simeq \{g\in \GL_2(\R)\mid \det\,g>0\}/\R^*.$

Computations and explanations are based on the Poincar\'e half-plane. The Poincar\'e disk model is only used for pictures.
The paper contains repeated fastidious computations involving products of two-by-two matrices, which have been checked with Python. Claude was used to draw the pictures.

Given a toplogical  surface $\cS$, a 
 faithfull representation 
$$\rho:\pi_1(\cS)\to\PSL_2(\R)$$
 is called {\it loxodromic} if its
image is a discrete subgroup of $\PSL_2(\R)$. Identifying
$\cS$ with $\HP/\rho(\pi_1(\cS)$ provides a hyperbolic metric
on $\cS$.

Any loxodromic representation can be lifted
to $\SL_2(\R)$, as  proved by L. Bers \cite{Bers}. 
However we should notice a strange phenomenon, already observed in \cite{Bers}. To keep track of the 
$\Z/3$-symmetry of the $3$-holed sphere, it is pleasant to present its fundamental group as
$$\langle x, y, z\mid xyz=1\rangle.$$
In a loxodromic representation, representing $x, y, z$ by matrices 
$X, Y, Z \in \SL_2(\R)$ with negative traces is the only way to maintains the relation
$XYZ=1$ together with the $\Z/3Z$-symmetry. This explains the 
annoying minus signs in the formulas of the paper.

\section{A surface $S$ of genus two}

The fundamental group  $\pi_1(\cS)$ 
of a closed oriented topological surface $\cS$ of genus $2$ is usually presented as
\begin{align}\langle x_1,x_2,y_1,y_2\mid 
(x_1,y_1)(x_2,y_2)=1\rangle \label{standard}\end{align}

\noindent
where  $(x,y):=xyx^{-1}y^{-1}$ denotes
the commutator. 

In the section, we describe an explicit representation of $\pi_1(\cS)$  in 
$\PSL_2(\R)$ and we consider the corresponding hyperbolic surface 
$S:=\HP/\pi_1(\cS)$.

\subsection{The Coxeter group  ${\bf W}\subset \Isom(\HP)$}

Let  $H\subset\HP$ be a regular right-angled hexagon. 
Let $(S_i)_{i\in\Z/6\Z}$ be the six sides of $H$, arranged in a cyclic order. Let $r_i$ be the
othogonal reflection along the geodesic containing $S_i$. By definition, each $r_i$ is a 
reversing-orientation isometry of the plane.

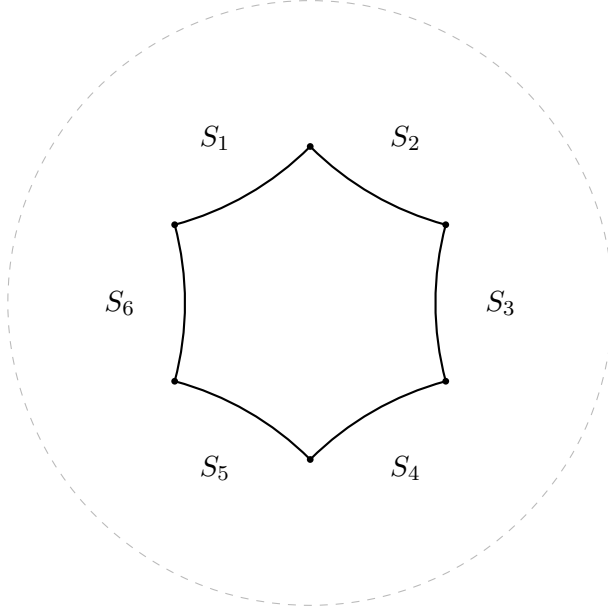
\begin{tikzpicture}[
    scale=4,
    vertex/.style={circle, fill=black, inner sep=0.9pt},
    lbl/.style={font=\small}
]

% ---- Regular right-angled hexagon, centred at the disc origin -------------
% Side length L : cosh L = 2.
% Hyperbolic circumradius R : cosh R = sqrt(3).
% Euclidean circumradius (Poincaré model) : r = sqrt(2 - sqrt 3) ~ 0.5176.
%
% Each side is a geodesic = circular arc perpendicular to the unit circle.
% Solving 2 C . v_k = 1 + |v_k|^2 for two adjacent vertices and using
% perpendicularity to the disc boundary gives, by 6-fold symmetry,
%   * arc-centre distance from origin : sqrt(2),
%   * arc radius                       : 1,
%   * each arc subtends 30 deg on its supporting circle.
%
% Convention used below:
%   - v_k placed at angle  90 - 60(k-1) from the origin (so v_1 is at the top,
%     v_2,...,v_6 go around clockwise);
%   - edge S_k is the side between v_{k-1} and v_k  (v_0 := v_6);
%   - the arc-centre supporting S_k sits at angle  120 - 60(k-1) from origin;
%   - on that circle v_k sits at angle  315 - 60(k-1)  and
%                    v_{k-1} at angle  285 - 60(k-1) ;
%     the CW arc from v_k to v_{k-1} sweeps 30 deg.
% ----------------------------------------------------------------------------

\pgfmathsetmacro{\R}{sqrt(2 - sqrt(3))}   % vertex Euclidean radius
\pgfmathsetmacro{\Cd}{sqrt(2)}             % arc-centre distance from origin

% ---- Poincaré disc boundary -----------------------------------------------
\draw[gray!55, dashed, thin] (0, 0) circle (1);

% ---- Vertices --------------------------------------------------------------
\foreach \k in {1, ..., 6} {
  \pgfmathsetmacro{\angV}{90 - 60*(\k-1)}
  \coordinate (v\k) at ({\R*cos(\angV)}, {\R*sin(\angV)});
}

% ---- Geodesic sides --------------------------------------------------------
\foreach \k in {1, ..., 6} {
  \pgfmathsetmacro{\angStart}{315 - 60*(\k-1)}
  \pgfmathsetmacro{\angEnd}{285 - 60*(\k-1)}
  \draw[thick] (v\k) arc[start angle=\angStart, end angle=\angEnd, radius=1];
}

% ---- Vertex dots ----------------------------------------------------------
\foreach \k in {1, ..., 6}
  \node[vertex] at (v\k) {};

% ---- Edge labels S_1 ... S_6 ----------------------------------------------
% Each label sits on the outward radial direction through the midpoint of
% the corresponding side, just past the arc.
\pgfmathsetmacro{\rlab}{0.63}
\foreach \k in {1, ..., 6} {
  \pgfmathsetmacro{\angL}{120 - 60*(\k-1)}
  \node[lbl] at ({\rlab*cos(\angL)}, {\rlab*sin(\angL)}) {$S_{\k}$};
}

\node[draw] at (0,-1.2) {The regular hexagon $H$ in the Poincar\'e disc};

\label{6gon}
\end{tikzpicture}

Let ${\bf W}$ the group generated by $(r_i)_{i\in\Z/6\Z}$.
The following lemma is well-known

\begin{lemma} 
The group ${\bf W}$ acts properly discontinuously
on $\HP$ and its fondamental domain is $H$.

Moreover, ${\bf W}$ is a Coxeter group with presentation
$$\langle (r_i)_{i\in\Z/6\Z} \mid 
r_i^2=
(r_i r_{i+1})^2=1\rangle.$$
\end{lemma}

\begin{proof} The statement is a consequence of Poincar\'e
Theorem \cite{Poincare}, where each side $S_i$ is paired with itself through the
reflexion $r_i$, see \cite{Maskit} for a  nice account of
Poincar\'e Theorem. It also follows from the theory of Coxeter groups, see
\cite{Davis}.
\end{proof}

For $w\in {\bf W}$, we define $l_0(w)$ and $l_1(w)$ as  the number of  even/odd letters in a reduced decomposition $w=r_{i_1}\dots r_{i_l}$ of $w$. Therefore,
$l(w):=l_0(w)+l_1(w)$ is the Bruhat length of $w$. Define the signature homomorphisms
$\epsilon_0, \epsilon_1:W\to\{\pm 1\}$ by 

\centerline{$\epsilon_i(w)=(-1)^{l_i(w)}$.}

Set  $\Pi:=\Ker\epsilon_0\cap\Ker\epsilon_1$ and $S=\HP/\Pi$.

\begin{lemma} \label{fundamental}

\begin{enumerate}
\item[(a)] Any isometry $w\in \Pi$ preserves the orientation,
\item[(b)] the subgroup $\Pi$ acts freely on $\HP$, and 
\item[(c)] a  fundamental domain for $\Pi$ is 
$$Q:= H\cup r_1(H)\cup r_2(H)\cup r_1r_2(H).$$
\end{enumerate}

Consequently, $S$ is a closed oriented surface of genus $2$, and
$\Pi\simeq \pi_1(S)$.
\end{lemma}

\begin{tikzpicture}[
    scale=4,
    vertex/.style={circle, fill=black, inner sep=0.9pt},
    lbl/.style={font=\small}
]

% ---- Geometry (regular right-angled hexagon, cosh L = 2) -------------------
% Vertex 1 at the centre of the Poincaré disc; interior edges are diameters.
% Midpoints (labels 2, 6) sit at Euclidean distance 1/sqrt(3) from the centre.
% Each outer geodesic side is a circular arc perpendicular to the unit circle.
%
%   long arc (between two 3's or two 5's):
%       centre on an axis at (0, 2/sqrt(3)) (and rotations),  radius 1/sqrt(3)
%   short arc (between a 3 and a 4, or a 5 and a 4):
%       centre at (1/sqrt(3), sqrt(3)/2) (and rotations),     radius sqrt(3)/6
%
% Intersecting the appropriate pairs perpendicularly yields the vertex
% positions in the upper-right quadrant:
%   V3 = (sqrt(3)/5, 2*sqrt(3)/5),
%   V4 = (1/sqrt(3), 1/sqrt(3)),
%   V5 = (2*sqrt(3)/5, sqrt(3)/5).
% ----------------------------------------------------------------------------

\pgfmathsetmacro{\aS}{atan(3/4)}        % small sweep  (~36.87 deg)
\pgfmathsetmacro{\aL}{atan(4/3)}        % large sweep  (~53.13 deg)

\pgfmathsetmacro{\midR}{1/sqrt(3)}      % midpoint radius
\pgfmathsetmacro{\arcR}{1/sqrt(3)}      % long-arc radius
\pgfmathsetmacro{\smcR}{sqrt(3)/6}      % short-arc radius
\pgfmathsetmacro{\Vthx}{sqrt(3)/5}      % V3 x
\pgfmathsetmacro{\Vthy}{2*sqrt(3)/5}    % V3 y
\pgfmathsetmacro{\Vfvx}{2*sqrt(3)/5}    % V5 x
\pgfmathsetmacro{\Vfvy}{sqrt(3)/5}      % V5 y

% ---- Poincaré disc boundary -----------------------------------------------
\draw[gray!55, dashed, thin] (0,0) circle (1);

% ---- Named coordinates ----------------------------------------------------
\coordinate (Ctr) at (0, 0);

\coordinate (M_T) at (0, \midR);
\coordinate (M_B) at (0, -\midR);
\coordinate (M_L) at (-\midR, 0);
\coordinate (M_R) at (\midR, 0);

\coordinate (Cne) at ( \midR,  \midR);
\coordinate (Cnw) at (-\midR,  \midR);
\coordinate (Cse) at ( \midR, -\midR);
\coordinate (Csw) at (-\midR, -\midR);

\coordinate (V3ne) at ( \Vthx,  \Vthy);
\coordinate (V3nw) at (-\Vthx,  \Vthy);
\coordinate (V3se) at ( \Vthx, -\Vthy);
\coordinate (V3sw) at (-\Vthx, -\Vthy);

\coordinate (V5ne) at ( \Vfvx,  \Vfvy);
\coordinate (V5nw) at (-\Vfvx,  \Vfvy);
\coordinate (V5se) at ( \Vfvx, -\Vfvy);
\coordinate (V5sw) at (-\Vfvx, -\Vfvy);

% ---- Outer boundary: 16 arcs (4 quadrants by rotation) --------------------
\foreach \rot in {0, 90, 180, 270} {
  \begin{scope}[rotate=\rot]
    \draw[thick] (0, \midR)
        arc[start angle=270, end angle=270+\aS, radius=\arcR];
    \draw[thick] (\Vthx, \Vthy)
        arc[start angle=180+\aS, end angle=270, radius=\smcR];
    \draw[thick] (\midR, \midR)
        arc[start angle=180, end angle=180+\aL, radius=\smcR];
    \draw[thick] (\Vfvx, \Vfvy)
        arc[start angle=180-\aS, end angle=180, radius=\arcR];
  \end{scope}
}

% ---- Interior cross (geodesics through the disc centre = diameters) -------
\draw[thick] (M_T) -- (M_B);
\draw[thick] (M_L) -- (M_R);

% ---- Vertices -------------------------------------------------------------
\foreach \p in {Ctr,
                M_T, M_B, M_L, M_R,
                Cne, Cnw, Cse, Csw,
                V3ne, V3nw, V3se, V3sw,
                V5ne, V5nw, V5se, V5sw}
    \node[vertex] at (\p) {};

% ---- Edge labels in the upper-right (NE) hexagon --------------------------
% Edge between vertex k and vertex (k+1) is S_{k+1}, so edge 6-1 -> S_1, etc.
% Outer-arc labels are placed at the arc midpoints, offset along the outward
% radial direction (from the corresponding arc centre toward the disc boundary
% side), so they sit outside the figure.
% ----------------------------------------------------------------------------

% Common offset (in disc-units) used to push outer-arc labels off the curve
\pgfmathsetmacro{\offa}{0.075}

% S_1: midpoint of inner edge Ctr -- M_R (label sits inside the NE hexagon)
\node[lbl] at ({\midR/2}, {0.07}) {$S_1$};

% S_2: midpoint of inner edge Ctr -- M_T  (label sits inside the NE hexagon)
\node[lbl] at ({0.07}, {\midR/2}) {$S_2$};

% S_3: outside the M_T -- V3ne arc
% arc centre (0, 2/sqrt 3), radius 1/sqrt 3, midway angle 270 + aS/2
\pgfmathsetmacro{\angA}{270 + \aS/2}
\pgfmathsetmacro{\mAx}{0       + \arcR*cos(\angA)}
\pgfmathsetmacro{\mAy}{2/sqrt(3) + \arcR*sin(\angA)}
% radial outward direction from disc origin
\pgfmathsetmacro{\rAn}{sqrt(\mAx*\mAx + \mAy*\mAy)}
\node[lbl] at ({\mAx + \offa*\mAx/\rAn}, {\mAy + \offa*\mAy/\rAn}) {$S_3$};

% S_4: outside the V3ne -- Cne arc
% arc centre (1/sqrt 3, sqrt 3/2), radius sqrt 3/6, midway angle 180+aS+aL/2
\pgfmathsetmacro{\angB}{180 + \aS + \aL/2}
\pgfmathsetmacro{\mBx}{1/sqrt(3) + \smcR*cos(\angB)}
\pgfmathsetmacro{\mBy}{sqrt(3)/2 + \smcR*sin(\angB)}
\pgfmathsetmacro{\rBn}{sqrt(\mBx*\mBx + \mBy*\mBy)}
\node[lbl] at ({\mBx + \offa*\mBx/\rBn}, {\mBy + \offa*\mBy/\rBn}) {$S_4$};

% S_5: outside the Cne -- V5ne arc
% arc centre (sqrt 3/2, 1/sqrt 3), radius sqrt 3/6, midway angle 180 + aL/2
\pgfmathsetmacro{\angC}{180 + \aL/2}
\pgfmathsetmacro{\mCx}{sqrt(3)/2 + \smcR*cos(\angC)}
\pgfmathsetmacro{\mCy}{1/sqrt(3) + \smcR*sin(\angC)}
\pgfmathsetmacro{\rCn}{sqrt(\mCx*\mCx + \mCy*\mCy)}
\node[lbl] at ({\mCx + \offa*\mCx/\rCn}, {\mCy + \offa*\mCy/\rCn}) {$S_5$};

% S_6: outside the V5ne -- M_R arc
% arc centre (2/sqrt 3, 0), radius 1/sqrt 3, midway angle 180 - aS/2
\pgfmathsetmacro{\angD}{180 - \aS/2}
\pgfmathsetmacro{\mDx}{2/sqrt(3) + \arcR*cos(\angD)}
\pgfmathsetmacro{\mDy}{0         + \arcR*sin(\angD)}
\pgfmathsetmacro{\rDn}{sqrt(\mDx*\mDx + \mDy*\mDy)}
\node[lbl] at ({\mDx + \offa*\mDx/\rDn}, {\mDy + \offa*\mDy/\rDn}) {$S_6$};

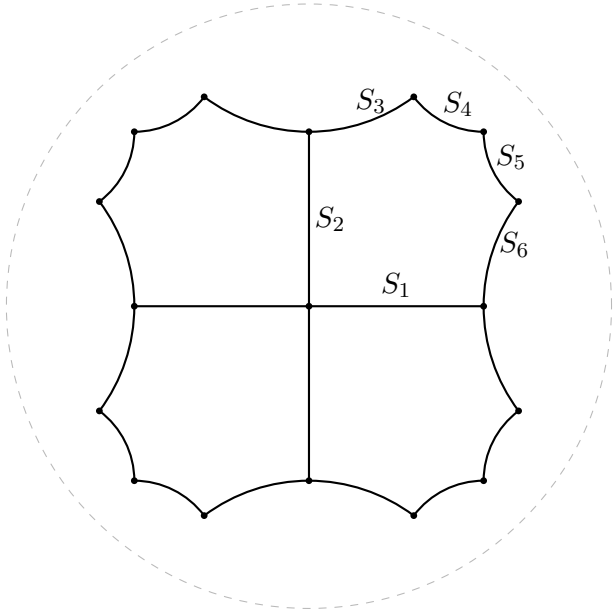
\node[draw] at (0,-1.2) {The fundamental domain $Q$};
\label{12gon}
\end{tikzpicture}

\begin{proof} Assertion (a) follows from the fact that
$\det w= (-1)^{l(w)}$ for all $w\in {\bf W}$.

In order to establish Assertion (b), we prove that any
isometry $w\in \Pi$ fixing one point $P\in \HP$ is the identity. Replacing $w$ by a conjugate, we can assume that $P$ lies in the fundamental hexagon $H$. 
We observe

\begin{enumerate}
\item[(a)]
The group ${\bf W}$ acts freely on the interior $H^0$ of $H$, so the statement is clear if $P$ lies in $H^0$. 

\item[(b)]
If $P$ lies in the interior 
of some side $S_i$, then $w$ lies in the group 
$\{1, r_i\}$. Since
$\Pi\cap \{1, r_i\}$, we conclude that $w=1$.

\item[(c)]
Otherwise $P$ is a vertex of $H$, that is 
$\{P\}=S_i\cap S_{i+1}$ for some $i\in\Z/6\Z$.  Therefore $w$ belongs to the subgroup $\langle r_i, r_{i+1}\rangle$. 
Since
$\Pi\cap \langle r_i, r_{i+1}\rangle=\{1\}$, we deduce as well that $w=1$.
\end{enumerate}

It follows that $S=\HP/\Pi$ is a closed oriented hyperbolic 
surface. By Gauss-Bonnet formula,
we have $\Area(H)=4\pi-6\pi/2=\pi$, therefore

$$\Area(S)=\Area(Q)=4 \Area(H)=4\pi.$$

\noindent It follows that  $S$ is a surface of genus two.
\end{proof}

For $i\in\Z/6\Z$, set $a_i=r_{i-1}r_{i+1}$. From now on, we will denote by
$1, 2,\ldots, 6$ the  representatives of $\Z/6Z$.

\begin{lemma}
 The subgroup $\Pi$ is generated by $a_1,\ldots, a_6$
 and satisfies the following relations
 \begin{align*}
 a_1 a_3 a_5&=1\\
 a_2 a_4 a_6&=1 \\
 a_1 a_2 a_3 a_4 a_5 a_6&=1 
 \end{align*}
\end{lemma}

\begin{proof}
The three relations are obvious, thus we only need to prove that
$\Pi$ is generated by the 
$a_i$, for $i\in\Z/6\Z$. 

Easy computations shows that 
\begin{align*}
r_ia_i r_i&=a_i\\
r_ia_{i\pm1}r_i&=a_{i\pm1}^{-1}\\
r_ia_{i+2}r_i&=a_{i+1}a_{i+2}a_{i+1}^{-1}\\
r_ia_{i-2}r_i&=a_{i-1}^{-1}a_{i-2}a_{i-1}\\
r_ia_{i+3}r_i&=a_{i+1}a_{i-1}=a_{i+1}a_{i+3}^{-1}a_{i+1}^{-1}
\end{align*}

It follows that the subgroup 
$A:=\langle a_i\mid i\in\Z/6\Z\rangle$ is invariant.
Since 
$$r_i\equiv r_{i+2}\equiv r_{i-2}\,\,\mod A,$$ 
the index of the sugroup $A\subset{\bf W}$ is $\leq 4$, therefore
$A=\Pi$, which proves that
$\Pi$ is generated by $a_1,\ldots, a_6$.
\end{proof}

\begin{lemma}\label{relations} The group $\Pi$ is generated by $a_1,\ldots, a_6$
and defined by the following relations

 \begin{align}
 a_1 a_3 a_5&=1 \label{free1}\\
 a_2 a_4 a_6&=1 \label{free2}\\
 a_1 a_2 a_3 a_4 a_5 a_6&=1 \label{relation}
 \end{align}
\end{lemma}

\begin{proof} Let $\Pi'$ be the group  generated by $a_1,\ldots, a_6$
and defined by Relations  \ref{free1}, \ref{free2} and \ref{relation}.
The first two 
relations simply
mean that
$\Pi'$ is generated by the four elements $a_1, a_2, a_4$ and $a_5$. Substituting $a_1^{-1}a_5^{-1}$ for $a_3$ and 
$a_4^{-1}a_2^{-1}$ for $a_6$ in Relation \ref{relation},
we conclude that $\Pi'$ is defined by the single relation

$$a_1a_2(a_1^{-1}a_5^{-1})a_4.a_5
(a_4^{-1}a_2^{-1})=1$$

\noindent which is equivalent to

$$(a_2^{-1}, a_1).(a_5^{-1}, a_4)=1.$$

\noindent Hence $\Pi'$ is isomorphic, as an abstract group,  to $\pi_1(S)$. By Lemma
\ref{relations}, there is  an epimorphism  $\pi:\Pi'\to \Pi$, sending each generator $a_i\in \Pi'$ to
$a_i\in\Pi$. By Lemma \ref{fundamental}, the group $\Pi$ is also isomorphic to $\pi_1(S)$.
Since $\pi_1(S)$ is finitely generated and residually finite, 
the group $\pi_1(S)$ is Hopfian \cite{Malcev}, thus
$\pi$ is an isomorphism, which proves the lemma.
\end{proof}

\section{The embedding $\Pi\subset \SL_2(\R)$}

By definition, the  group $\Pi$ is a subgroup
of $\Isom^+(\HP)\simeq \PSL_2(\R)$. In fact, by Bers Theorem
\cite{Bers}, $\Pi$ can be realized as a subgroup of
$\SL_2(\R)$. 
In order to describe an explicit
embedding $\Pi\subset\SL_2(\R)$, we first define  a certain central extension
$\widehat{\bf W}$ of $\bf{W}$ in the group
$\GL_2(\R)$.

\subsection{The embedding $\widehat{\bf W}\subset \GL_2(\R)$}

In order to explicitely describe $\bf{W}$ as a subgroup of $\PGL_2(\R)$,
we first describe an explicit embedding of $H$ in $\HP$.
Let $\Omega$ be the center of the right-angled hexagon $H$,
let $L$ be the length of its sides $S_i$  and let $M$ be the center of the side $S_1$.

\begin{lemma} We have

\begin{enumerate}
\item[(a)] $\cosh\,L=2$,
\item[(b)] the geodesic arc $\Omega M$  is orthogonal to $S_1$,
\item[(c)] $d_\HP(\Omega, M)=\ln(1+\sqrt{2})$.
\end{enumerate}
\end{lemma}

\begin{proof} Assertion (b) follows easily from the symmetries of $H$. 

Let $P$ be one extremity of $S_1$. The triangle 
$P\Omega M$ has angles $\pi/4, \pi_6$ and $\pi/2$.
We use the short notation $PM$ for $d_\HP(P,M)$ and 
$\Omega M$ for $d_\HP(\Omega,M)$. By elementary trigonometry, we have
$$ \cosh(L/2)=\cosh(PM)=
\frac{\cos(\pi/6)}{\sin(\pi/4)}=\frac{\sqrt{6}}{2}$$

\noindent and therefore
$$\cosh(L)=2\cosh^2(L/2)-1=2,$$
which proves the Assertion (a).

Similarly, we obtain
$$
\cosh(\Omega M)=\frac{\cos(\pi/4)}{\sin(\pi/6)}=\sqrt{2}$$
Since 
$$\cosh(\ln(1+\sqrt{2}))= 
\frac{1}{2} (1+\sqrt{2}+(1+\sqrt{2})^{-1})=\sqrt{2}$$
we deduce $d_\HP(\Omega, M)=\ln(1+\sqrt{2}).$
\end{proof}

We observe that 
$$d_\HP(i,(1+ \sqrt{2})i)=\ln(1+\sqrt{2}).$$
By the previous lemma,  there is an embedding of $H$ in $\HP$, such that
\begin{enumerate}
\item[(a)] $H$ is centered around $i$, 
\item[(b)] $(1+\sqrt{2})i$ is the middle of the side $S_1$.
\end{enumerate}

Since $S_1$ and $\R i$ are orthogonal, $S_1$ lies
in the euclidian circle centered at $0$ with euclidian
radius $1+\sqrt{2}$.
It follows that the inversion $r_1$ can be represented by the matrix

$$R_1:=\begin{pmatrix}
0&1+\sqrt{2} \\
 \sqrt{2}-1&0
\end{pmatrix}.$$

\noindent Since $\det R_1<0$, this  representation means that $r_1(z)=\frac{(3+2\sqrt{2})}{\overline z}$.

The hyperbolic rotation of angle $\pi/3$ around $i$ 
is  represented by the matrix

$$G:=\begin{pmatrix}
\frac{\sqrt{3}}{2} & -\frac{1}{2}\\
\frac{1}{2} & \frac{\sqrt{3}}{2}
\end{pmatrix}$$

\noindent Indeed we have $G^6=-1$, that is $G$ has order $6$
in $\PSL_2(\R)$.

 For all $i\in\Z/6\Z$, set  $R_i=G^{1-i}R_1G^{i-1}$. By construction, each $R_i$ is the matrix representation of the inversion $r_i$.  Explicitely we have

$$R_2:=\begin{pmatrix}
\frac{\sqrt{6}}{2}& \frac{2+\sqrt{2}}{2}\\
\frac{-2+\sqrt{2}}{2} &-\frac{\sqrt{6}}{2}
\end{pmatrix},
R_3:=\begin{pmatrix}
\frac{\sqrt{6}}{2}& \frac{2-\sqrt{2}}{2}\\
-\frac{2+\sqrt{2}}{2} &-\frac{\sqrt{6}}{2}
\end{pmatrix}
$$

$$R_4:=\begin{pmatrix}
1-\sqrt{2} & 0\\
0 &-(1+ \sqrt{2})
\end{pmatrix},
R_5:=\begin{pmatrix}
-\frac{\sqrt{6}}{2}& \frac{2-\sqrt{2}}{2}\\
-\frac{2+\sqrt{2}}{2} &\frac{\sqrt{6}}{2}
\end{pmatrix}, \text{ and}$$

$$R_6:=\begin{pmatrix}
-\frac{\sqrt{6}}{2}& \frac{2+\sqrt{2}}{2}\\
-\frac{2-\sqrt{2}}{2} &\frac{\sqrt{6}}{2}
\end{pmatrix}.$$

Let $\widehat{\bf W}$ be the subgroup of $\GL_2(\R)$ generated by  
$R_1,\ldots,R_6$.

\begin{lemma} The subgroup $\widehat{\bf W}\subset\GL_2(\R)$  is a central extension of
$\bf{W}$ by $\{\pm 1\}$. It satisfies
\begin{align}
R_i^2&=1\label{cover1}\\
R_iR_{i+1}&=-R_{i+1}R_i\label{cover2}
\end{align}
\end{lemma}

\begin{proof} By definition, $R_i$ is a matrix representation of the reflection along
the side $S_i$ of $H$. Thus the image of $\widehat{W}$ in $\PGL_2(\R)$ is the Coxeter group 
$\bf {W}$, which implies that $\widehat{\bf W}$  is a central extension of
$\bf{W}$.

It is clear that $R_1^2=1$ and an explicit computation shows that 
$R_1R_2=-R_2R_1$. Formulas \ref{cover1} and \ref{cover2} are obtained by conjugation by $G$.
\end{proof}

\subsection{The embedding $\Pi\subset \SL_2(\R)$}

Set $A_i:=R_{i-1}R_{i+1}$ for any $i\in\Z/6\Z$. Since $\det R_i=-1, \forall i\in\Z/6\Z$,
we deduce that all matrices $A_i$ are in $\SL_2(\R)$.

\begin{lemma} The map $a_i\mapsto A_i$ extends to a group isomorphism

 $$\Pi\simeq \langle A_i\mid i\in\Z/6\Z\rangle\subset \SL_2(\R).$$
\end{lemma}

\begin{proof} It follows from Relation \ref{cover1} that 
$A_1A_3A_5=1$ and $A_2A_4A_6=1$.

We have 

$$A_1A_2A_3A_4A_5A_6=(R_6R_2)(R_1R_3)(R_2R_4)(R_3R_5)(R_4R_6)(R_5 R_1)$$

We observe that $R_i$ commutes with $R_{i-1}R_{i+1}$.
Using that 
\begin{align*}
R_i(R_{i-1}1R_{i+1})R_i&=R_{i-1}1R_{i+1}\\
R_i^2&=1
\end{align*}

 we deduce

\begin{align*}
A_1A_2A_3A_4A_5A_6&=(R_6R_2)(R_1R_3)(R_2R_4)(R_3R_5)(R_4R_6)(R_5 R_1)\\
&=R_6\underline{R_2R_1R_3R_2}R_4R_3R_5R_4R_6R_5 R_1\\
&=R_6R_1R_3\underline{R_4R_3R_5R_4}R_6R_5 R_1\\
&=R_6R_1\underline{R_3R_3}R_5R_6R_5 R_1 \\
&=\underline{R_6R_1R_5R_6}R_5 R_1\\
&=R_1\underline{R_5R_5} R_1\\
&=\underline{R_1^2}\\
&=1
\end{align*}

\noindent where, in each line of the previous equation, we have underlined the subword used to
simplify the expression. Thus, by Lemma \ref{relations}, 
the  subgroup $\langle A_i\mid i\in\Z/6\Z\rangle$ is isomorphic to
$\Pi$.
\end{proof}

\section{The explicit quaternion ring of $S$}

The surface $\HP/\Pi$ has been investigated 
in \cite{KuuNaa} who have noticed that this surface is arithmetic. In fact, this follows also from
\cite{Takeuchi}. In this section, we describe the 
corresponding quaternion ring.

\subsection{Explicit expressions for the $A_i$.}

Since $A_i=R_{i-1}R_{i+1}$, we deduce from the
formulas for $R_i$ that:

$$A_1=\begin{pmatrix}
-2 &-\sqrt{6}-\sqrt{3}\\
-\sqrt{6}+\sqrt{3}&-2
\end{pmatrix},
A_2=\begin{pmatrix}
-2-\frac{3\sqrt{2}}{2}&-\frac{\sqrt{6}}{2}-\sqrt{3}\\
-\frac{\sqrt{6}}{2}+\sqrt{3}&-2+\frac{3\sqrt{2}}{2}\end{pmatrix}$$

$$A_3=\begin{pmatrix}
-2-\frac{3\sqrt{2}}{2}&\frac{\sqrt{6}}{2}-\sqrt{3}\\
\frac{\sqrt{6}}{2}+\sqrt{3}&-2+\frac{3\sqrt{2}}{2}\end{pmatrix},
A_4=\begin{pmatrix}
-2 &\sqrt{6}-\sqrt{3}\\
\sqrt{6}+\sqrt{3}&-2
\end{pmatrix}$$

$$A_5=\begin{pmatrix}
-2+\frac{3\sqrt{2}}{2}&\frac{\sqrt{6}}{2}-\sqrt{3}\\
\frac{\sqrt{6}}{2}+\sqrt{3}&-2-\frac{3\sqrt{2}}{2}\end{pmatrix},
A_6=\begin{pmatrix}
-2+\frac{\sqrt{2}}{2} &-\frac{\sqrt{6}}{2}-\sqrt{3}\\
-\frac{\sqrt{6}}{2}+\sqrt{3}&-2-\frac{\sqrt{2}}{2}
\end{pmatrix}$$

Let $\Quat(3,3)$ be the quaternion algebra over $\R$ with presentation

$$\langle I, J\mid I^2=3, J^2=3, IJ+JI=0\rangle.$$

\begin{lemma}\label{quaternion} We have
\begin{enumerate}
\item[(a)] $(2+A_1)^2=3$,
\item[(b)] $(2+A_2)^2=3$, and
\item[(c)] $(2+A_1)(2+A_2)+(2+A_2)(2+A_1)=0$.
\end{enumerate}
Consequently, there is an isomorphism 
$\Mat_2(\R)\simeq \Quat(3,3)$ relative to which
$I=2+A1$ and $J=2+A2$.
\end{lemma}

\begin{proof} The Assertions (a), (b) and (c) follow from direct computation. Consequently, there is an algebra homomorphism
$\Quat(3,3)\to \Mat_2(\R)$ sending $I$ to $(2+A_1)$ and
$J$ to $(2+A_2)$. Since the quaternion algebra $\Quat(3,3)$ is simple, this algebra homomorphism is an isomorphism.
\end{proof}

As usual, set $K:=IJ$. Let $\cH\subset \Quat(3,3)$ be the subring

$$\cH:=\Z\oplus\Z I\oplus\Z J\oplus \Z K.$$

Using the explicit formulas for the matrices $A_i$, we easily deduce the following lemma:

\begin{lemma} Relative to the isomorphism of the previous lemma, we have:

$A_1=-2+I$

$A_2=-2+J$

$A_3=-2-2I+K$

$A_4=-2-3I-2J+2K$

$A_5=-2-2I-3J+2K$

$A_6=-2-2J+K$

In particular $\Pi$ lies in $\cH$.
\end{lemma}

The reduced  trace $\Tr h$ and the reduced determinant $\det h$ of an arbitrary element  $h=a+BI+cJ+dK\in\cH$ are given by

$$\Tr h= 2a,$$

$$\det h= a^2 -3b^2 -3 c^2 +9d^2.$$

\begin{lemma} \label{H*}

\begin{enumerate}
\item[(a)] We have $\det h=1$  for any $h\in\cH^*$.
\item[(b)] $\cH^*$ contains neither elliptic elements nor
parabolic elements.
\item[(c)] In particular, $\Q\otimes\cH$ is not split.
\end{enumerate}

\end{lemma}

\begin{proof} Let $h=a+bI+cJ+dK$ be an invertible element
in $\cH$. 

We have  $\det h=\pm1$. Since 
$\det h=a^2-3b^2-3c^2+9d^2$, we have

$$\det h\equiv a^2 \mod 3$$

\noindent which implies that $\det h=1$, proving Assertion (a).

We also observe that $a\neq 0$. Since $\Tr\, h=2a$, we
conclude that $\mid\Tr\, h\mid\geq 2$ therefore
$h$ is not elliptic.

We observe that the equation
$-b^2-c^2+3 d^2=0$ admits no nontrivial integer solution.
Otherwise we can pick a primitive solution, that is a solution
with $\gcd(a,b,c)=1$. The equation
$b^2+c^2\equiv 0\mod 3$ implies that $b$ and $c$ are divisible by $3$.
Therefore $3 d^2=a^2+b^2$ is divisible by $9$, thus $c$ is also divisible by
$3$ which contradicts that $\gcd(b,c,d)=1$. Therefore
any element $h$ with $\Tr(h)=\pm 2$ is equal to $\pm1$,
which means that $\cH^*$ contains no parabolic elements,
which completes the proof of Assertion (b).

The last assertion follows from the fact that  $\cH^*$ contains no parabolic elements.
\end{proof}

\begin{cor} \label{Bers}We have
$$\cH^*=\{ \pm1 \}\times \Pi$$
\end{cor}

\begin{proof} 
It follows from Lemma \ref{H*} that $\cH^*/\{\pm1\}$ is isomorphic to
the fundamental group of a closed oriented surgface of genus $g$ for a certain $g\geq 2$.
By Hurwitz formula applied to the Galois cover
$\HP/\Pi\to \HP/\cH^*$, we have

$$2-1 =[\cH^*/\{\pm1\}:\Pi] (g-1),$$

\noindent which implies that $\cH^*/\{\pm1\}=\Pi$. The claim follows.
\end{proof}

\subsection{Cardinality of the  group $\Gamma_q$}

For an integer $q\geq 2$, set
$\cH_q=\cH/q\cH$ and

$$\Gamma_q=\{h\in\cH_q\mid \det h=1\mod\,q\}.$$

\noindent In the subsection, we determine the cardinality of the group $\Gamma_q$, i.e. 
the number of solutions of the equation in the unknowns
$a, b, c, d\in \Z/q\Z$

\begin{align}
a^2-3b^2-3c^2+9d^2&=1 \mod\, q.\label{normeq}
\end{align}

We start with the  case $q=2^n$. In order to avoid
repetitions in the proof, we consider 
the more general equation in the unknowns
$x_1, x_2, x_3, x_4\in \Z/q\Z$

\begin{align}
a_1x_1^2+a_2x_2^2+a_3x_3^2+a_4x_4^2=1\,\mod\,q,\label{normeq2}
\end{align}

\noindent where $a_1,\ldots, a_4$ are four 
given integers congruent to $1$ modulo $4$.
Let $\Sol(q)$ be the set of solutions of Equation \ref{normeq2}
and let $\theta:\Sol(2q)\mapsto \Sol(q)$ be the natural map
sending $(x_1,x_2,x_3,x_4) to (x_1,x_2,x_3,x_4) \text{ modulo } 2q$.

\begin{lemma} \label{sol2} We have
$$\Card\Sol(q)=q^3.$$
\end{lemma}

\begin{proof}
We first consider the case $q=2^n$ with $n\leq 2$.
We observe that the equation \ref{normeq2}, namely 
$x_1^2+x_2^2+x_3^2+x_4^2=0$, admits a symmetry of order $4$. 

\begin{enumerate}
\item[(a)] For $n=1$, the only solutions, up to the $\Z/4\Z$-symmetry, are
$1,0,0,0)$ and $(1,1,1,0)$, thus
  $\Card\Sol(2)=8=2^3.$
  
\item[(b)]
For $n=2$, we observe that the square modulo four
are $0$ and $1$. Thus, up to the $\Z/4\Z$-symmetry, the only solutions
are $(\pm1, 2y_2, 2y_3, 2 y_4)$ where $y_i=0$ or $1$.
Thus  $\Card\Sol(4)=4\times 16=4^3.$
\end{enumerate}

We now consider the case $q=2^n$ for $n\geq 3$. 
Let $(x_1,x_2,x_3,x)\in 
\Sol(q)$.
By the previous consideration, exactly one
coordinate $x_i$ is odd. Thus 
$\Sol(q)$ is the disjoint union:

$$\Sol(q)=\Sol_1(q)\cup\Sol_2(q)\cup\Sol_3(q)\cup\Sol_4(q)$$

\noindent where 

$$\Sol_i(q):=\{(x_1,x_2,x_3,x)\in 
\Sol(q)\mid\,x_i\,\text{is odd and } x_j \text{ is even for}
j\neq i\}.$$

We will now compute $\Card\Sol_1(q)$. We observe
that an odd element $x\in \Z/q\Z$ is a square
if and only if $q\equiv 1\mod 8$. Furthermore,  when the condition is satisfied, $x$ admits exactly four square roots.

Next, when $x$ is even, we have $x^2\equiv 2 x\mod 8$.
Thus $\Sol_1(q)$ is the set of $4$-uples 
$(x_1,x_2,x_3,x)\in\Z/q\Z^4$ satisfying

\begin{enumerate}
\item[(a)] $x_2, x_3$ and $x_4$ are even,
\item[(b)] $1/a_1\big(1-2(a_2x_2+a_3x_3+a_4x_4)\big)\equiv 1\mod 8$
\item[(c)] $x_1$ is a square root of 
$1/a_1\big(1-2(a_2x_2+a_3x_3+a_4x_4)\big)$.
\end{enumerate}

There are $(q/2)^3$ even triples $(x_2, x_3,x_4)$, and
the quantity 
$1/a_1\big(1-2(a_2x_2+a_3x_3+a_4x_4)\big)$ is congruent  to $1$ or $5$ modulo $8$. Thus there are $1/2 (q/2)^3$
triple $(x_2, x_3,x_4)$ satisfying (a) and (b). We deduce

$$\Card\Sol_1(q)=4\times 1/2(q/2)^3=q^3/4.$$

Similarly, $\Card\Sol_i(q)=q^3/4$ for $i=2,3, 4$, which proves that 
$$\Card\Sol(q)=q^3.$$
\end{proof}

Recall that an integer $q$ of the form $q=p^n$, where 
$p$ is prime and $n\geq 1$ is called a
{\it prime power}.

\begin{lemma}\label{solq} Let $q=p^n$ be a prime power

\begin{enumerate}
\item[(a)] If $p=2$, then $\Card \Gamma_q=q^3$.
\item[(b)] If $p=3$, then $\Card \Gamma_q=2q^3$,
\item[(c)] Otherwise, we have 
$\Card \Gamma_q=(1-1/q^2) q^3$.
\end{enumerate}

\end{lemma}

\begin{proof}  The proof amonts to counting 
the number of solutions of Equation \ref{normeq}.

If $p=2$, the formula is proved by Lemma \ref{sol2}.

If $p=3$, we observe that an element
$x\in (\Z/q\Z)^*$ is a square if and only if
$x\equiv 1\mod 3$. Thus $b$, $c$ and $d$ can be arbitrary
and equation
\ref{normeq}  can be rewritten as

$$a=\pm\sqrt{1+3b^2+3d^2-9d^2},$$

\noindent which proves that $\Card \Gamma_q=2q^3$.

If $p\geq 5$, the quaternion ring $\cH_q$ is isomorphic to
$\Mat_2(\Z/q\Z)$, therefore 
$$\Gamma_q\simeq \SL_2(\Z/q\Z),$$
which implies the formula.
\end{proof}

\begin{cor}\label{corcardinal} For an arbitrary integer $q\geq 2$, we have

$$\Card (\Gamma_q)\leq 2 q^3.$$
\end{cor}

\begin{proof} Let $q=q_1\ldots q_l$
be the decomposition of $q$ into prime powers, with
$\gcd(q_i,q_j)=1$ for $i\neq j$.
By the Chinese remainder theorem, we have
$$\Gamma_q\simeq \Gamma_{q_1}\times\ldots\times \Gamma_{q_1}.$$

\noindent Henceforth, by Lemma \ref{solq}, we have
$\Card  \Gamma_q\leq 2 q^3$. 
\end{proof}

\section{Schmutz Theorem}

\subsection{Systoles and minimal traces}

Let $\cS$ be a closed oriented topological surface.
A {\it loop} is a continous map 
$f[0,1]\to{\cS}$ such that $f(0)=f(1)$.
We precise that
\begin{enumerate}
\item[(a)] our loops are unoriented, that is
we identify the loop
$f:t\mapsto f(t)$ with its inverse
$g: t\mapsto f(1-t)$. 
\item[(b)] we do not impose a fixed based point.
\end{enumerate}

By definition, a {\it curve} is the homotopy class
of an unbased loop $f:S^1\to {\cS}$ which is not null-homotopic.
The fundamental group $\pi_1(\cS)$  of the surface 
$\cS$ is defined up to an inner automorphism. The {\it non oriented conjugacy class}
of an element $a\in \pi_1(\cS)\setminus\{1\}$ is the set
$$[a]:=\{ga^{\pm 1}g^{-1}\mid g\in \pi_1(\cS)\}.$$
Elementary topological considerations shows
that the curves of $\cS$ and the nonoriented conjugacy classes are in a one-by-one correspondance. We will
denote as $c(a)$ the curve corresponding with the nonoriented conjugacy class $[a]$.

Given a loxodromic representation
$$\rho:\pi_1(\cS)\to\PSL_2(\R)\simeq \Isom^+(\HP),$$
the quotient $\cS(\rho):= \HP/\rho(\pi_1(\cS))$ is an hyperbolic surface homeomorphic to
$\cS$. Any curve $c$ of $\cS(\rho)$
is homotopic to a unique geodesic. 
In fact, for  $a\in \pi_1(\cS)$ the {\it axis}  of $\rho(a)$ is the complete geodesic of $\HP$  whose  endpoints are the fixed points of  
$\rho(a)$ on the boundary $\partial \HP:=\R\cup{\infty}$.
Then the geodesic $c(a)$ 
is the image  of any segment of the axis of length
$\arcosh(\frac{\mid\Tr(\rho(a))}{2})$.  Equivalently
$$1/2\mid \Tr(\rho(a)\mid=\cosh(l(c(a))/2)\,$$
where $l(c(a))$ is the length of  $c(a)$, see
\cite{Buser}.
 
\noindent We deduce the 
following well-known consequence:

\begin{cor}\label{mintrace} The systoles of $\cS(\rho)$ are the curves $c(a)$ where $[a]$ is a nonoriented conjujugacy class such that $\mid \Tr(\rho(a)\mid$ is minimal.

Moreover the systole length is 
$2\arcosh(1/2\mid \mid\Tr(\rho(a)\mid)$, where $a$ is any element with  trace of minimal absolute value.
\end{cor}

\subsection{Schmutz Theorem}

Set $S=\HP/\Pi$. By lemma, all elements in
$h=\cH^*\setminus\{1\}$ are of the form
$a+bI+cJ+dK$, where $\mid a\mid $ is an integer $\geq 2$.
It follows that 

$$\mid\Tr\,h\mid =2\mid a\mid \geq 4.$$
It follows that  geodesics $c_1:=c(a_1),\ldots, c_6:=c(a_6)$ are
systoles. Indeed Schmutz has proved that there are no other systoles \cite{SS}.

\begin{Schmutz} \label{schmutz} The geodesics $c_1,\ldots, c_6$ are the six systoles of $S$.
\end{Schmutz}

We deduce the following algebraic corollary.

\begin{cor}\label{corSchmutz} Let $h\in\cH^*$ be an element
with $\mid\Tr\,h\mid=4$. 

Then $h$ or $-h$ is conjugate to some 
$a_i^{\pm1}$.
\end{cor} 

\begin{proof}  By Lemma $h$ or $\-h$ belongs to $\Pi$,
so we can assume $h\in\Pi$.
By the previous considerations, $c(h)$ is
a systole. Therefore by
Schmutz Theorem \ref{schmutz},
$h$ is conjugate to some 
$a_i^{\pm1}$.
\end{proof}

\section{The systoles of the arithmetic surfaces $S(n)$}

For any $n\geq 1$, let $p_n$ and $q_n$ 
be the integers defined by the equation

\begin{align}
(2+\sqrt{3})^n&=p_n+q_n\sqrt{3}.\label{defpq}
\end{align}

\noindent The natural
ring homomorphism 
$$\cH\to \cH_{q_n}, h\mapsto h\,\mod\,q_n$$
 induces a group homomorphism

$$\pi_n: \Pi\to \Gamma_{q_n}/\{\pm1\}.$$

\noindent Set $\Pi(n):=\Ker\,\pi_n$ and
$S(n):=\HP/\Pi(n).$ We will also denote as

$$\pi_n:S(n)\to S$$

\noindent the corresponding Galois covering of $S$ with Galois group $\Pi/\Pi(n)$. In this section,
we will determine the systoles of $S(n)$.

\subsection{Some preparatory lemmas}

Set $\alpha=2+\sqrt{3}$. Obviously, we have
$\alpha^{-1}=2-\sqrt{3}$.
We start with an easy lemma.

\begin{lemma}\label{formulas} The two eigenvalues of
$A_i^{\pm 1}$ are $-\alpha$ and $-\alpha^{-1}$.

Moreover we have

\begin{align}
\alpha^{-n}&=p_n-q_n\sqrt{3}\label{formulaminus}\\
1&=p_n^2-3q_n^2 \label{formuladet}\\
p_n&= \frac{1}{2}(\alpha^{n}+\alpha^{-n}) \label{formulap}\\
q_n&= \frac{1}{2\sqrt{3}}(\alpha^{n}-\alpha^{-n}).  \label{formulaq}
\end{align}
\end{lemma}

\begin{proof} We have

$$-4=\Tr A_i=-\alpha-\alpha^{-1},$$ 

\noindent therefore 
the eigenvalues of
$A_i^{\pm 1}$ are $-\alpha$ and $-\alpha^{-1}$.

Obviously
$\alpha^{-1}=2-\sqrt{3}$ is conjugate to $\alpha$ by the involution
of $\Gal(\Q(\sqrt{3})/\Q)$. Thus we deduce  
Formula \ref{formulaminus}. Formula
\ref{formuladet} follows from the identity
$$1=\alpha^n\alpha^{-n}=p_n^2-3q_n^n.$$

Combining Equation \ref{defpq} and Formula \ref{formulaminus}, we deduce 
Formulas \ref{formulap} and \ref{formulaq}.
\end{proof}

\begin{lemma}\label{nthroot} Let $h\in\SL_2(\R)$ with
$\frac{1}{2}\Tr\,h=p_n$. Set

$$k=\frac{h-p_n}{q_n}.$$

Then we have 
$$(2+k)^n=h.$$
\end{lemma}

\begin{proof}  We will use the formulas 
\ref{formulap}  and \ref{formulaq}  of the previous lemma. 

Formula \ref{formulap} implies that  the eigenvalues of $h$ are $\alpha^{\pm n}$.
Up to conjugacy, we can assume that

$$h=\begin{pmatrix}
\alpha^n & 0\\
0 &\alpha^{-n}
\end{pmatrix}.$$

It follows easily that 

$$h-p_n=\begin{pmatrix}
\frac{1}{2}(\alpha^n-\alpha^{-n}) & 0\\
0 &\frac{1}{2}(\alpha^{-n}-\alpha^n)
\end{pmatrix},$$

Since by Formula \ref{formulaq}
$q_n= \frac{1}{2\sqrt{3}}(\alpha^{n}-\alpha^{-n})$
we have

$$k=\begin{pmatrix}
\sqrt{3} & 0\\
0 &-\sqrt{3}
\end{pmatrix},$$

\noindent and therefore

$$2+k=\begin{pmatrix}
\alpha& 0\\
0 &\alpha^{-1}
\end{pmatrix},$$

\noindent which implies that $(2+k)^n=h$. 
\end{proof}

\subsection{The key lemma}

The proof of Theorem \ref{main} is based on the following
easy key lemma.

\begin{lemma}\label{key} Let $h\in\Pi(n)\setminus\{1\}$.

\begin{enumerate} 
\item[(a)] We have $\frac{1}{2}\vert \Tr\,h\vert \geq p_n$.
\item[(b)]  Moreover, if $\frac{1}{2}\vert \Tr\,h\vert = p_n$,
then  $h$ is $\Pi$-conjugate 
to some $a_i^{\pm1 n}$.
\end{enumerate}
\end{lemma}

\begin{proof} By definition, we have 
$p_n^2=1+3q_n^2$.

Any element $h\in \Pi(n)\setminus\{1\}$ is of the form

$$h=p+q_n (bi+cj+dk),$$ 
where $p, b, c, d$ are integers.
Since $\det h=1$, we have

$$p^2=1+3eq_n^2,$$ 
where $e=b^2+c^2-3d^2$ is a non zero integer. It follows easily that 

$$\mid\Tr h\mid= 2\mid p\mid \geq p_n,$$

which proves the first assertion.

Let $h'\in\{h,-h\}$ be the element such that
$\Tr h'=\,p_n$. Set $k'=\frac{h'-p_n}{q_n}$.
By Lemma \ref{nthroot}, we have

$$(2+k')^n=h'.$$ 

\noindent Since $h$ belongs to $\Pi(n)$, the element
$k'=\frac{h'-p_n}{q_n}$ belongs to $\cH$, thus
$(2+k')$ belongs to $\cH^*$. Obviously
we have $\Tr(2+k)=4$. By Corollary \ref{corSchmutz},
$2+k$ is conjugate by the group $\Pi$ to  
some $-a_i^{\pm1}$. It follows that
$h'$ is $\Pi$-conjugate to some 
$\pm a_i^{\pm n}$. By Corollary \ref{Bers},
 $-h$ does not belongs to $\Pi$, thus 
 $h$ is $\Pi$-conjugate to some 
$a_i^{\pm n}$.
\end{proof}

\subsection{The systoles of $S(n)$}

Recall that $c_1,\ldots, c_6$ are the six systoles of $S$ of length $L=2\arcosh(2)=2\ln(\alpha)$, where
$\alpha=2+\sqrt(3)$.  Also recall that $\pi_n:S(n)\to S$ is a Galois covering with Galois group $\Pi/\Pi(n)$. 

\begin{thm}\label{systolesSn} Let $n\geq 2$.
\begin{enumerate}
\item[(a)] For $i=1,\ldots, 6$,  
$\pi_n^{-1}(c_i)$ is a disjoint union of 
$\frac{1}{n} [\Pi:\Pi(n)]$ geodesics of length
$nL$.
\item[(b)] The systoles of $S(n)$ are exactly the
connected components of $\pi_n^{-1}(c_i)$ for
$i=1,\ldots,6$.
 \end{enumerate}
\end{thm} 

\begin{proof} By Lemma \ref{key}, elements 
$h\in\Pi(n)\setminus\{1\}$ with $\mid\Tr\,h\mid$ minimal have trace $\pm 2 p_n$. 
Thus by Corollary \ref{mintrace} the lenght of systoles
of $S(n)$ is $2\arcosh(p_n)$. Since by  Formula \ref{formulap}  of Lemma 
\ref{formulas}

$$p_n=\frac{1}{2}(\alpha^n+\alpha^{-n})$$

we deduce that length  of the systoles of $S(n)$ is

$$2\ln(\alpha^n)=nL.$$

Since $a_i^n$ belongs to $\Pi(n)$, we conclude that
$c(a_i^n)$ is a systole of $S(n)$. Hence
$\pi_n^{-1}(c_i)$ consists of $c(a_i^n)$ and all its conjugate by the Galois group $\Pi/\Pi(n)$.
Since the total length of $\pi^{-1}(c_i)$
is $[\Pi:\Pi(n)]s$ and each connected component has length
$ns$, we conclude that
$\pi^{-1}(c_i)$ consists of $\frac{1}{n} [\Pi:\Pi(n)]$ are systoles of length $ns$, which proves the first assertion.

Let $h\in\Pi\setminus\{1\}$ be an element such that
$c(h)$ is a systole of $S(n)$. By the previous considerations, $\mid\Tr\,h\mid=p_n$. Thus 
by Lemma \ref{key}, $h$
is conjujugate by $\Pi$ to some $a_i^{\pm n}$ for
some $i\in\Z/6\Z$, which amonts to the fact that $c(h)$ is a connected component of $\pi_n^{-1}(c_i)$. This concludes
the proof. \end{proof}

\section{Proof of the Theorem}
We conclude the paper with the proof of the Theorem.
First, we show that the number of systoles of $S(n)$ is
$$\leq 18\ln(\alpha)\frac{g_{S(n)}}{\ln g_{S(n)}}.$$
Then we consider the case where $n$ is even. Then using
\cite{Mathieu} and \cite{AII}, we show the existence of a
hyperbolic surface $\Sigma\in\cP_g$ near $S(n)$ such that
$\Syst(\Sigma)$ has cardinality 
$\frac{1}{2}\Card \Syst(S(n))$, which proves the Theorem.

\subsection{Estimate of the genus $g_{S(n)}$ of $S(n)$}

Let $g_{S(n)}$ be the genus of the surface $S(n)$.

\begin{lemma}\label{genus} For $n\geq 2$, we have
$$g_{S(n)}<\alpha^{3n}$$
\end{lemma}

\begin{proof}  The surface $S$ has genus two and
the Galois cover $\pi_n:S(n)\to S$ has degree 
$[\Pi:\Pi(n)]$. Thus by Hurwitz formula, we have

$$g_{S(n)}=1+[\Pi:\Pi(n)].$$

Since $\Pi/\Pi(n)$ is a subgroup of
$\Gamma_{q_n}$, by Corollary \ref{corcardinal}, we have
$[\Pi:\Pi(n)]\leq 2q_n^3$. By Formula \ref{formulaq} of Lemma \ref{formulas}, we have
$q_n\leq \frac{\alpha^n}{2\sqrt{3}}$. It follows easily that

 $$g_{S(n)}=1+[\Pi:\Pi(n)]\leq 1+2q_n^3<\alpha^{3n}.$$
\end{proof}

\begin{cor}\label{estimate} For $n\geq 2$, we have
$$\Card(\Syst(S(n))\leq 18\ln(\alpha)\frac{g_{S(n)}}{\ln g_{S(n)}}.$$
\end{cor}

\begin{proof}  By Theorem 
\ref{systolesSn},
each set $\pi_n^{-1}(c_i)$ contains exactly 
$\frac{1}{n} [\Pi:\Pi(n)]$ systoles.
Therefore 
$\Card(\Syst(S(n))\leq\frac{6}{n} [\Pi:\Pi(n)]$ systoles.

By Lemma \ref{genus}, we have $g_{S(n)}<\alpha^{3n}$. We deduce

$$\frac{1}{n}<\frac{3\ln(\alpha)}{\ln(g_{S(n)})}.$$

Using that 

$$[\Pi:\Pi(n)]=g_{S(n)}-1<g_{S(n)},$$

we deduce that

$$\frac{6}{n} [\Pi:\Pi(n)]
<18 \ln(\alpha)\frac{g_{S(n)}}{\ln(g_{S(n)})},$$

\noindent which completes the proof.
\end{proof}

\subsection{The surface $S(2)$}

We observe that $H_1(\pi_1(S))\simeq \Z^4$, therefore there is a unique
Galois covering of $S$ of Galois group $(\Z/2\Z)^4$.

\begin{lemma} The surface $S(2)$ is the surface of genus $17$
which is a Galois covering of $S$ of Galois group $(\Z/2\Z)^4$.

In particular, $\Syst(S(2))$ contains exactly $48$ systoles,
which are the connected components of $\pi_2^{-1}(c_i)$ for $i=1,\ldots, 6$.
\end{lemma}

\begin{proof} We observe that $q_2=4$. Recall that 
$$\cH_4=\cH/4\cH \text{ and }\Gamma_4=\{h\in \cH_4^*\mid \det(h)=1 \mod 4\}.$$
By Assertion (a) of Lemma \ref{solq}, $\cH_4^*$ is a 2 group.

We show that $\Gamma_4/\{\pm 1\}$ is commutative. First we observe
that 
$$ak=ka=-ka\, \,\,\text{ for all } a\in \cH_4 \text{ and } k\in 2 \cH_4.$$
Any element in $h\in \Gamma_4$ is of the form $a+k$, where
$a\in \{\pm 1, \pm I, \pm J, \pm K\}$ and $k\in  2 \cH_4$.
Let $h'=a'+k'$ be another element in $\Gamma_4$.
We have 
\begin{align*} aa'&=\pm a'a,\\
ak'&=k'a =-k'a, \\
a'k&=ka'=-ka'\\
kk&=k'k=0
\end{align*}

 It follows that $hh'=\pm h'h$ which 
proves that
$\Gamma_4/\{\pm 1\}$ is commutative.

Similarly, we show that $h^2=\pm 1$ for any $h\in \Gamma_4$.
We deduce that $\Pi/\Pi(4)\simeq (\Z/2\Z)^r$ for some $r$.
Since $a_1,\ldots, a_6$ have distinct images in
$\Gamma_4/\{\pm 1\}$, we deduce that $2^r\geq 6$. Since
$\Pi/(\Pi,\Pi)= H_1(S)\simeq \Z^4$, we finally deduced that
$\Pi/\Pi(4\simeq (\Z/2\Z)^4$, which completes the proof.

By Theorem \ref{systolesSn}, $\pi_2^{-1}(c_i)$ is a union of $8$ systoles of length $2L$. Therefore $\Syst(S(2))$ consists of exactly $48$ systoles. \end{proof}

\bigskip
As it is shown by figure \ref{12gon},  $c_3\cup c_4\cup c_5\cup c_6$ cuts $S$ into
a single $12$-gon. Thus the set of $32$ systoles
$$\{ c\in \Syst(S(2))\mid c\subset
\pi_2^{-1}(c_3\cup c_4\cup c_5\cup c_6)\}$$ 
fills $S(2)$.  In \cite{Mathieu} it was hinted that this set contains a filling subset with only 25 systoles. 

However An, Irhinger and Irmer obtained a better result:

\begin{AII}\cite{AII}\label{AIIthm} The set $\Syst(S(2))$ contains a filling subset $A$ of $24$ elements.
\end{AII}

\subsection{Hexagonal regular Penner systems}

Let $X$ be a hyperbolic surface, and let $C$ be a set of simple geodesics. Following \cite{Penner},
the set $C$ is called a {\it Penner system} if 
 
\begin{enumerate}
\item[(a)] $C$ fills $\Sigma$, 
\item[(b)] $C$ is partitionned into $A=B\cup R$, where geodesics in $B$ are called {\it blue}
and geodesics in $R$ are called {\it red}, and 
\item[(c)] geodesics of the same colour do not intersect.
\end{enumerate}

The Penner system $C$ is called {\it hexagonal} if it cuts $X$ into regular right-angled
hexagons. The following theorem has been proved in \cite{Mathieu}

\begin{thm}\label{Mat} Let $X$ be a hyperbolic surface. Assume that $\Syst(X)$ is a hexagonal 
Penner system.

Then for any filling subset $B\subset \Syst(X)$, there is a hyperbolic surface
$\Sigma$ near $X$ in $\cT_g$ such that $\Syst(\Sigma)=B$.
\end{thm}

\subsection{Proof of the theorem}

The main result is an obvious consequence of the following
result.

\begin{thm} For any $n\geq 1$, there is a hyperbolic surface 
$\Sigma$ of genus $g=g_{S(2n)}$ which is filled by its systoles such that

$$\Card(\Syst(\Sigma))<9\ln(2+\sqrt{3})\frac{g}{\ln\,g}.$$
\end{thm}

\begin{proof} Since
$$q_{n+2}=4 q_{n+1}-q_n,$$
it follows  that $q_{2n}$ is divisible by $q_2=4$. Hence 
$\Pi(2n)\subset \Pi(2)$ and there is a Galois covering

$$\theta_n: S(2n)\to S(2)$$

\noindent with Galois group $\Pi(2)/\Pi(2n)$. By Theorem \ref{AIIthm},
there is a filling subset $A\subset \Syst(S(2))$ of cardinality
$\frac{1}{2} \Card \,\Syst(S(2))$. Set

$$B=\{c\in\Syst(S(2n))\mid \theta_n(c)\in A\}.$$

The set $\Syst(S(2n))$ is a Penner system, where a systole is blue if it lies  inside $\pi_n^{-1}(c_1\cup c_3\cup c_5)$ and  is red otherwise.
By Theorem \ref{Mat}, there is a hyperbolic surface near $S(2n))$ with systole set
$B$. It is clear that $B$ fills $\Sigma$, and its cardinality is
$\frac{1}{2} \Card \,\Syst(S(2n))$. Therefore by Corollary \ref{estimate}, we obtain

$$\Card(\Syst(\Sigma)<9\ln(2+\sqrt{3})\frac{g}{\ln\,g},$$

\noindent which completes the proof.
\end{proof}

%\bibliography{QuaternionBib}

\end{document}